\documentclass{amsart}
\usepackage{amsmath, amssymb, amsthm, amsfonts}
\usepackage{etoolbox} 

\theoremstyle{plain}
\newtheorem{theorem}{Theorem}
\newtheorem{prop}{Proposition}[section]

\newtheorem{lemma}[prop]{Lemma}

\theoremstyle{definition}

\AtEndEnvironment{definition}{\null\hfill$\diamond$}%
\newtheorem{remark}[prop]{Remark}
\AtEndEnvironment{remark}{\null\hfill$\diamond$}%

\AtEndEnvironment{example}{\null\hfill$\diamond$}%

\newcommand{\NN}{\mathbb{N}}

\newcommand{\ZZ}{\mathbb{Z}}
\newcommand{\RR}{\mathbb{R}}

\newcommand{\Lo}{\operatorname{Lo}}
\newcommand{\lo}{\operatorname{\mathfrak{lo}}}

\newcommand{\SubMatrix}{\operatorname{SubMatrix}}
\newcommand{\nsc}{\operatorname{nsc}}
\newcommand{\nz}{\operatorname{nz}}

\newcommand{\rank}{\operatorname{pr}}

\newcommand{\frakl}{\mathfrak{l}}

\newcommand{\sign}{\operatorname{sign}}

\newcommand{\Pos}{\operatorname{Pos}}

\newcommand{\cC}{\mathcal{C}}

\newcommand{\bM}{{\mathbf M}}

\newcommand{\bR}{\mathbb{R}}
\newcommand{\ee}{\end{equation}}

\title[Finiteness of rank for Grassmann convexity]
{Finiteness of rank for Grassmann convexity}

\author{Nicolau C. Saldanha}
\address{ Departamento de Matem\'atica, PUC-Rio, R. Mq. de S. Vicente 225, Rio de Janeiro, RJ 22451-900, Brazil}
\email[N. Saldanha]{saldanha@puc-rio.br}

\author{Boris Shapiro}
\address{Department of Mathematics, Stockholm University, SE-106 91, Stockholm, Sweden}
\email[B. Shapiro]{shapiro@math.su.se}
\author{Michael  Shapiro}
\address{Department of Mathematics, Michigan State University,
East Lansing, MI 48824-1027 and National Research University Higher School of Economics, Moscow, Russia}
\email[M. Shapiro]{mshapiro@msu.edu}

\thanks{The first author is supported  by CNPq, Capes and Faperj (Brazil). The third author is partially supported by International Laboratory of Cluster Geometry NRU HSE, RF Government grant, ag. № 075-15-2021-608 dated 08.06.2021 and the NSF grant DMS 2100791.} 


\begin{document}

\begin{abstract}
The  Grassmann convexity conjecture, formulated in \cite{ShSh},
gives a conjectural formula  for the maximal  total number of real zeros
of the consecutive Wronskians 
of an arbitrary fundamental solution
to a disconjugate linear ordinary differential equation with real time.
The conjecture can be reformulated in terms of  convex curves
in the nilpotent lower triangular group.
The formula has already been shown to be a correct lower bound
and to give a correct upper bound in several small dimensional cases.
In this paper we obtain a general explicit upper bound.
\end{abstract}


\maketitle

\section{Introduction}
\label{section:introduction}

A linear homogeneous differential equation of order $n$
with continuous  real coefficients defined on some open interval $I$
of the real axis is called \emph{disconjugate} on $I$
(see \cite{Coppel, Hartman})
if any of its nonzero solutions has at most $n-1$ real zeros on $I$,
counting multiplicities.
Any such differential equation is disconjugate
on a sufficiently small interval $I\subset \bR$.

The  Grassmann convexity conjecture, formulated in \cite{ShSh},
says that, for any positive integer $k$ satisfying $0<k<n$,
the number of real zeros in $I$ of the Wronskian
of any $k$ linearly independent solutions
to a disconjugate differential equation of order $n$ on $I$ 
is bounded from above by $k(n-k)$,
which is the dimension of the Grassmannian $G(k,n;\,\bR)$. 
 This conjecture can be reformulated in terms of  convex curves
in the nilpotent lower triangular group
(see Section~\ref{section:discrete} and also \cite{GoSaX, GoSa0, GoSa1}).
The number $k(n-k)$ has already been shown to be a correct lower bound for all $k$ and $n$. Moreover
it gives a correct upper bound for the case $k=2$ 
(see \cite{SaShSh}).
In this paper we obtain a general explicit upper bound 
for the above maximal total number of real zeros which,
although weaker than the one provided
by the  Grassmann convexity conjecture,
still gives an interesting information. 

\smallskip
We now formulate a discrete problem which is equivalent
to the above Grassmann convexity conjecture; 
the equivalence is discussed in Section \ref{section:discrete}. 
The statement and proof of our main theorem 
are self-contained and elementary. 

Given integers $k$ and $n$ with $0 < k < n$,
consider a collection $(v_j)_{1 \le j \le n}$ of vectors in $\RR^k$,
or, equivalently, a matrix $M \in \cC = \RR^{k\times n}$. 
Set $v_j = Me_j$, so that $v_j$ is the $j$-th column of $M$.
For two matrices $M_0, M_1 \in \cC$, we say that 
there exists a \textit{positive elementary move} from $M_0$ to $M_1$
if there exists $j < n$ and $t \in (0,+\infty)$ such that
$M_1 e_j = M_0 e_j + t M_0 e_{j+1}$ and
$M_1 e_{j'} = M_0 e_{j'}$ for $j' \ne j$.
In other words, the vector $v_j$ moves towards $v_{j+1}$
and the other vectors remain constant.
We call a  sequence $(M_s)_{0 \le s \le \ell}$
of $k\times n$-matrices \textit{convex}  if, for all $s < \ell$,
there exists a positive elementary move from $M_s$ to $M_{s+1}$. 
The terminology is motivated by their relationship with convex curves
in the lower triangular group $\Lo_n^{1}$;
see Section \ref{section:discrete} and \cite{SaShSh}.

For $J = \{j_1 < \cdots < j_k\} \subset \{1, \ldots, n\}$, 
define the function $m_J: \cC \to \RR$ by 
\begin{equation}
\label{equation:mJ}
m_J(M) = \det(\SubMatrix(M,\{1,\ldots,k\},J)) =
\det(v_{j_1}, \ldots, v_{j_k}). 
\end{equation}
We are particularly interested in $m_\bullet = m_{\{1, \ldots, k\}}$.
Define  an open dense subset of $(k\times n)$-matrices 
\[ \cC^\ast = \bigcap_{J \subset \{1, \ldots, n\}, |J| = k}
m_J^{-1}[\RR \smallsetminus \{0\}] \subset \cC. \] 
For convex sequences $\bM = (M_s)_{0 \le s \le \ell}$
of matrices in $\cC^\ast$,
we are interested in the \textit{number of sign changes} of $m$:
\[ \nsc(\bM) = |\{ s \in [0,\ell) \cap \ZZ \;|\;
m_\bullet(M_s) m_\bullet(M_{s+1}) < 0 \}|. \] 
(Notice that $\bM$ is merely a sequence of matrices,
not an actual path of matrices;
equivalently, $s$ assumes integer values only.) 

The Grassmann convexity conjecture,
in the discrete model,
states that for every convex sequence $\bM$ we have:  
\begin{equation}
\label{equation:grassmann}
\nsc(\bM) \le k(n-k).
\end{equation}
The equivalence between formulations is
discussed in Section \ref{section:discrete}.
The case $k = 1$ is easy; 
the case $k = 2$ is settled in \cite{SaShSh};
in the same paper we prove that for all $k$ and $n$, 
there exists a convex sequence $\bM$ with $\nsc(\bM) = k(n-k)$.
In this paper we prove the following  estimate.

\begin{theorem}
\label{theorem:main}
For any\; $2 < k < n$ and  any convex sequence of matrices $\bM$,  we have
\begin{equation}\label{eq:main}
 \nsc(\bM) < \frac{(n-k+1)^{2k-3}}{2^{k-3}}.
 \end{equation} 
\end{theorem}

Inequality \eqref{eq:main} appears to be
the first known explicit upper bound
for $\nsc(\bM)$ in terms of $k$ and $n$.
This bound is a polynomial in $n$ of degree $2k-3$.  
Furthermore, by a well-known (Grassmann) duality interchanging 
$k \leftrightarrow (n-k)$, one can additionally obtain 
\[ \nsc(\bM) <
\min\left(2^{3-k} (n-k+1)^{2k-3},\; 2^{3-n+k} (k+1)^{2(n-k)-3}\right). \]
We know however that the bound in Theorem \ref{theorem:main} is  not  sharp;
see Remark \ref{remark:loose}.
Our current best guess is that the Grassmann convexity conjecture
indeed holds for all $0<k<n$;
this is additionally supported by
our recent computer-aided verification of the latter conjecture in case  $k=3, n=6$. 

\bigskip

\section{Discrete and continuous versions}
\label{section:discrete}

In this section we present an alternative continuous version 
of the Grassmann convexity conjecture
already discussed in previous papers
and prove that it is equivalent to the discrete version
described in the introduction.
We follow the notations of \cite{GoSa0} and \cite{SaShSh}.

Consider the nilpotent Lie group of $\Lo_{n}^{1}$
of real lower triangular matrices with diagonal entries equal to $1$.
Its corresponding Lie algebra is $\lo_{n}^1$,
the space of strictly lower triangular matrices.
For $J \subset \{1, \ldots, n\}$, $|J| = k$,
define $m_J: \Lo_{n}^{1} \to \RR$ 
as in Equation \eqref{equation:mJ};
thus, $m_\bullet(L) = m_{\{1, \ldots, k\}}(L)$
is the determinant of the  lower-left  $k\times k$ minor of $L$.
For $1 \le j < n$,
let $\frakl_j = e_{j+1} e_j^\top \in \lo_{n}^1$
be the matrix with only one nonzero entry $(\frakl_j)_{j+1,j} = 1$.
Write $\lambda_j(t) = \exp(t\frakl_j)$.

A smooth curve $\Gamma: [a,b] \to \Lo_{n}^{1}$
is \textit{convex} if there exist smooth positive functions
$\beta_j: [a,b] \to (0,+\infty)$ such that, for all $t \in [a,b]$,
\[ (\Gamma(t))^{-1} \Gamma'(t) = \sum_j \beta_j(t) \frakl_j. \]
We prove in \cite{GoSa0} that all zeroes of
$m_\bullet \circ \Gamma: [a,b] \to \RR$
are isolated (and of finite multiplicity).
Let $\nz(\Gamma)$ be the number of zeroes of $m_\bullet \circ \Gamma$
 counted  without multiplicity. 
The Grassmann convexity conjecture
states that $\nz(\Gamma) \le k(n-k)$
for all convex $\Gamma$.  (In fact, following the ideas presented in the proof of Theorem 2 of \cite{SaShSh}, one can show that the maximal numbers of zeros of $m_\bullet \circ \Gamma$ counted with and without multiplicity actually coincide.)    

Let $S_n$ be the symmetric group with the standard generators
$a_j = (j,j+1)$, $1 \le j < n$.
Let $\eta \in S_n$ be  
the permutation with the longest reduced word; its
 length  equals $m = n(n-1)/2$.
Let $\Pos_\eta \subset \Lo_{n}^{1}$ be the semigroup
of totally positive matrices, an open subset
(see \cite{BFZ, GoSa0}). 
The boundary of $\Pos_\eta$ can be stratified as
\[ \partial\Pos_\eta =
\bigsqcup_{\sigma \in S_n, \; \sigma \ne \eta} \Pos_\sigma \]
(we follow the notation of \cite{GoSa0};
the subsets $\Pos_\sigma \subset \Lo_n^1$ 
and the above stratification are discussed in Section 5).
Let $\sigma = a_{i_1} \cdots a_{i_l}$ be a reduced word:
if $L \in \Pos_\sigma$ then there exist unique positive
$t_1, \ldots, t_l$ such that 
$L = \lambda_{i_1}(t_1) \cdots \lambda_{i_l}(t_l)$.
Conversely, if $t_1, \ldots, t_l > 0$ then
$\lambda_{i_1}(t_1) \cdots \lambda_{i_l}(t_l) \in \Pos_\sigma$.
If $\Gamma$ is convex and $t_0 < t_1$ then
$(\Gamma(t_0))^{-1} \Gamma(t_1) \in \Pos_\eta$.
Conversely, if $L_0^{-1} L_1 \in \Pos_\eta$
then there exists a convex curve $\Gamma: [0,1] \to \Lo_n^1$
with $\Gamma(0) = L_0$ and $\Gamma(1) = L_1$
(see \cite{GoSa0}, Lemma 5.7).

\begin{lemma}
\label{lemma:discrete}
Fix $k, n, r \in \NN$, with $0 < k < n$.
If $\nsc(\bM) \le r$ for every
convex sequence $\bM = (M_s)_{0 \le s \le \ell}$, 
then $\nz(\Gamma) \le r$ for every convex curve $\Gamma$.
Conversely, 
if $\nz(\Gamma) \le r$ for every convex curve $\Gamma$
then
$\nsc(\bM) \le r$ for every
convex sequence $\bM = (M_s)_{0 \le s \le \ell}$.
\end{lemma}

\begin{proof}
Consider a convex curve $\Gamma_0$ with $\nz(\Gamma_0) = r$.
We use $\Gamma_0$ to construct a convex sequence $\bM$
with $\nsc(\bM) \ge r$.
Indeed, let $t_1 < \cdots < t_r$ be such that
$m_\bullet(\Gamma(t_s)) = 0$ for all $1 \le s \le r$.
Take $\tilde M_{s} = \SubMatrix(\Gamma(t_s),
\{n-k+1,\ldots, n\}, \{1, \ldots, n\})$
and corresponding vectors $\tilde v_{s,j}$.
 By taking a  small perturbation we may assume that  $m_J(\tilde M_s) \ne 0$ for $|J| = k$, $J \ne \{1, \ldots, k\}$.  
For $m = n(n-1)/2$ and  a small positive number $\epsilon > 0$,
set
$M_{(m+1)s} := \tilde M_s \lambda_k(-\epsilon)$,
$M_{(m+1)s+1} := \tilde M_s \lambda_k(\epsilon)$.
Notice that there exists a positive elementary move from
$M_{(m+1)s}$ to $M_{(m+1)s+1}$
and that $\sign(m_\bullet(M_{(m+1)s+1})) \ne \sign(m_\bullet(M_{(m+1)s}))$.
If the above  perturbation and $\epsilon > 0$ are sufficiently small,
there exists $L_s \in \Pos_\eta$ such that
$M_{(m+1)(s+1)} = M_{(m+1)s+1} L_s$.
Write $L_s = \lambda_{i_1}(t_1) \cdots \lambda_{i_m}(t_m)$
and, for $1 \le j \le m$, recursively define 
$M_{(m+1)s+j+1} = M_{(m+1)s+j} \lambda_{i_j}(t_j)$.
This is the desired convex sequence of matrices.

Conversely, let $\bM = (M_s)_{1 \le s \le \ell}$
be a convex sequence of matrices
with $\nsc(\bM) = r$.
We use $\bM$ to construct a smooth convex curve $\Gamma$
with $\nz(\Gamma) \ge r$.
For each $s$, we have
$M_{s+1} = M_s \lambda_{i_s}(t_s)$, $t_s > 0$.
Notice that $\lambda_{i_s}(t_s) \in \Pos_{a_{i_s}}$.
If needed slightly perturb the matrices $(M_s)$ to obtain matrices
$(\tilde M_s)$ such that
$\tilde M_{s+1} = \tilde M_s L_s$, $L_s \in \Pos_\eta$.
Define $\Gamma(1) \in \Lo_{n}^1$ such that
$\tilde M_{1} = \SubMatrix(\Gamma(1),
\{n-k+1,\ldots, n\}, \{1, \ldots, n\})$.
Recursively define $\Gamma(s+1) = \Gamma(s) L_s$
so that for all $s \in \ZZ$, $1 \le s \le \ell$, we have
$\tilde M_{s} = \SubMatrix(\Gamma(s),
\{n-k+1,\ldots, n\}, \{1, \ldots, n\})$.
Notice that for each $s$, there exists a smooth convex arc from $\Gamma(s)$ to $\Gamma(s+1)$.
As discussed in Section 6 of \cite{GoSa1},
these arcs can be chosen so that $\Gamma$
is also smooth at the integer glueing points.
\end{proof}

\bigskip

\section{Ranks}
\label{section:ranks}

For $0 < k < n$,
define $r(k,n) \in \NN \cup \{\infty\} = \{0, 1, 2, \ldots, \infty \}$ as 
\[ r(k,n) = \sup_{\bM} \nsc(\bM); \]
where $\bM$ runs over all convex sequences of matrices. 
By the main results of \cite{SaShSh},  $r(2,n) = 2(n-2)$;
and, additionally for all $0 < k < n$, we have $r(k,n) \ge k(n-k)$.
Using the examples for which $\nsc(\bM) = k(n-k)$ provided by \cite{SaShSh},
the Grassmann convexity conjecture
(as in Equation \eqref{equation:grassmann})
is equivalent to
\begin{equation}
\label{equation:grassmann2}
r(k,n) = k(n-k).
\end{equation}

\smallskip

Given a subset of $k\times n$-matrices $X \subseteq \cC^\ast$,
a \textit{prerank function} for $X$ is a function
$$\rank: X \to \NN = \{0, 1, 2, \ldots \}$$ with the  properties:
\begin{enumerate}
\item{if
there exists a positive elementary move from $M_0 \in X$ to $M_1 \in X$
then $\rank(M_0) \ge \rank(M_1)$;}
\item{if
there exists a positive elementary move from $M_0 \in X$ to $M_1 \in X$
and $m(M_0) m(M_1) < 0$
then $\rank(M_0) > \rank(M_1)$.}
\end{enumerate}
When $X$ is not mentioned we assume that $X = \cC^\ast$.
A prerank function for a convex sequence $\bM$
is, by definition,  a prerank function for its image.
A prerank function $\rank$ for $\cC^\ast$ is called \textit{regular}
if $\rank(MD) = \rank(M)$ for every $M \in \cC^\ast$
and every positive diagonal matrix $D \in \RR^{n \times n}$;
in terms of sequences of vectors,
this means that multiplying each vector $v_j$
by a positive real number $d_j$ does not affect the rank.  For $k = 2$, 
 a regular prerank function
$\rank: \cC^\ast \to [0,2(n-2)] \cap \NN$ has been constructed in \cite{SaShSh}.

\begin{lemma}
\label{lemma:rank}
For $0 < k < n$, the following properties are equivalent:
\begin{enumerate}
\item{$r(k,n) \le r_0$;}
\item{there exists a regular prerank function for $\cC^\ast$
whose image is contained in $[0,r_0]$;}
\item{there exists a prerank function for $\cC^\ast$
whose  image is contained in $[0,r_0]$;}
\item{for every convex sequence $\bM$ of matrices,  
there exists a prerank function for $\bM$
whose image is contained in $[0,r_0]$.}
\end{enumerate}
\end{lemma}

\begin{proof}
Assuming the first item, let us construct a regular prerank function.
For $M \in \cC^\ast$, set
\[ \rank(M) = \max_{\bM = (M_s)_{0 \le s \le \ell}, \; M_0 = M} \nsc(\bM). \]
Regularity follows from the observation that
if $D \in \RR^{n\times n}$ is a positive diagonal matrix
and $(M_s)$ is a convex sequence then so is $(M_s D)$.
The second item trivially implies the third.
 Assuming the third item,  we obtain a prerank function for $\cC^\ast$.
To prove the fourth item, given a convex sequence, restrict the above prerank function 
to the image of $\bM$ to obtain a prerank function for $\bM$,

Given a convex sequence $\bM = (M_s)_{0 \le s \le \ell}$
and a prerank sequence for $\ell$,  we obtain that 
$\nsc(\bM) \le \rank(M_0) - \rank(M_\ell)$.
Thus, the fourth item implies the first, completing the proof.
\end{proof}

\bigskip

\section{Step lemma}
\label{section:step}

The following lemma provides  
the induction step
necessary to settle Theorem \ref{theorem:main}.

\begin{lemma}
\label{lemma:step}
For all\;   $2 < k < n$,  we have that 
\[ r(k,n) \le \frac{(n-k+1)^2}{2} \, r(k-1,n-1). \]
\end{lemma}


\begin{proof}[Proof of Theorem \ref{theorem:main}]
In order to use induction on $k$,
notice first that the inequality \eqref{eq:main} holds for $k = 2$.
By inductive assumption and  Lemma \ref{lemma:step}, 
we get  
\begin{align*} 
\nsc(\bM) &\le r(k,n) \le
\frac{(n-k+1)^2}{2}  r(k-1,n-1) < \\
&<
\frac{(n-k+1)^2}{2}\,2^{3-(k-1)} (n-k+1)^{2(k-1) - 3} =
2^{3-k} (n-k+1)^{2k-3},
\end{align*}
completing the proof.
\end{proof}

\begin{proof}[Proof of Lemma \ref{lemma:step}]
Set $r_{-} = r(k-1,n-1)$ (assumed to be finite) and define 
 $r_{0} := (n-k+1)^2  r_{-}/2$.
Let $\cC^\ast_{-} \subset \cC_{-} = \RR^{(k-1)\times (n-1)}$
be the set of matrices $M \in \cC_{-}$ with
$m_J(M) \ne 0$ for all $J \subset \{1, \ldots, n-1\}$, $|J| = k-1$.
By Lemma \ref{lemma:rank}, there exists
a regular prerank function
$\rank_{-}: \cC_{-} \to [0,r_{-}] \cap \NN$.
Given a convex sequence $\bM = (M_s)_{0 \le s \le \ell}$ of matrices
in $\cC^\ast \subset \cC = \RR^{k \times n}$, 
we construct a prerank function for $\bM$
with image contained in $[0,r_0]$:
by Lemma \ref{lemma:rank} this will complete the proof of Lemma \ref{lemma:step}.

Define $v_{s,j} = M_s e_j$.
Let $H_0 \subset \RR^k$ be a linear hyperplane in generic position,
defined by a linear form $\omega: \RR^k \to \RR$;
fix a basis of $H_0$ in order  to identify it with $\RR^{k-1}$.
Let $H_1 = \omega^{-1}[\{1\}]$ be an affine hyperplane parallel to $H_0$.
We may assume that for all $s$ and $j$, we have
$\omega(v_{s,j}) \ne 0$.
We may furthermore assume that
$\omega(v_{0,j}) > 0$ for $j > n-k$,
which implies
$\omega(v_{s,j}) > 0$ for $j > n-k$ and all $s$.
Set $\tilde v_{s,j} = v_{s,j}/\omega(v_{s,j})$
so that $\tilde v_{s,j} \in H_1$.
For $j < n$, set
$w_{s,j} = \omega(v_{s,j+1}) (\tilde v_{s,j+1} - \tilde v_{s,j}) \in H_0$.
For  a given $s$,
the sequence of vectors $(w_{s,j})_{1 \le j \le n-1}$
defines 
a matrix $\tilde M_s \in \cC_{-} = \RR^{(k-1)\times (n-1)}$.
We assume that $m(M_s) \ne 0$, which implies $m(\tilde M_s) \ne 0$.

\smallskip
Define the prerank as
$\rank(M_s) = \rank_{I}(M_s) +  \rank_{II}(M_s)\cdot r_{-} $,  where 
\[ \rank_{I}(M_s) = \rank_{-}(\tilde M_s), \quad
\rank_{II}(M_s) = \sum_{1 \le j < n}
[ \omega(v_{s,j}) \omega(v_{s,j+1}) < 0 ]\cdot  j. \]
Here we use Iverson notation:
$[ \omega(v_{s,j}) \omega(v_{s,j+1}) < 0 ]$ equals $1$
if $\omega(v_{s,j}) \omega(v_{s,j+1}) < 0$ and $0$ otherwise.
In particular, if $j > n-k$ then  
$[ \omega(v_{s,j}) \omega(v_{s,j+1}) < 0 ] = 0$.
We therefore have
\[ \rank_{I}(M_s) \in [0,r_{-}] \cap \ZZ, \qquad
\rank_{II}(M_s) \in \left[0,\frac{(n-k+1)(n-k)}{2} \right] \cap \ZZ, \]
and thus 
\[ \rank(M_s) \in [0,r_0] \cap \ZZ. \]
We need to  verify that $\rank(M_s)$ is indeed a prerank function.

Recall that there exists a positive elementary move from $M_s$ to $M_{s+1}$,
which we call the \emph{$s$-th move} in $\bM$.
There are two kinds of positive elementary moves.
If $\sign(\omega(v_{s,j})) = \sign(\omega(v_{s+1,j}))$
for all $j$ then we say the $s$-th move is of type I;
otherwise it is of type II.

First consider $s$ of type II.
Let $j$ be such that $v_{s+1,j} = v_{s,j} + t v_{s,j+1}$, $t > 0$.
By taking $H$ in general position
and introducing, if necessary, intermediate points
we may assume that $\sign(m(M_s)) = \sign(m(M_{s+1}))$.
For $j' \ne j$, we have $v_{s+1,j'} = v_{s,j'}$
implying that  $\sign(\omega(v_{s+1,j'})) = \sign(\omega(v_{s,j'}))$.
Therefore  we have
\[ \sign(\omega(v_{s+1,j})) = \sign(\omega(v_{s+1,j+1})) =
\sign(\omega(v_{s,j+1})) = -\sign(\omega(v_{s,j})) \]
and thus
\[ [ \omega(v_{s,j}) \omega(v_{s,j+1}) < 0 ] = 1,
\qquad [ \omega(v_{s+1,j}) \omega(v_{s+1,j+1}) < 0 ] = 0. \]
For $j' = j-1$, we get
\[ [ \omega(v_{s,j'}) \omega(v_{s,j'+1}) < 0 ] =
1 - [ \omega(v_{s+1,j'}) \omega(v_{s+1,j'+1}) < 0 ]; \]
while for $j' \notin \{ j-1, j\}$,  we get 
\[ [ \omega(v_{s,j'}) \omega(v_{s,j'+1}) < 0 ] =
[ \omega(v_{s+1,j'}) \omega(v_{s+1,j'+1}) < 0 ]. \]
Thus, in all cases we obtain that 
$\rank_{II}(M_{s+1}) < \rank_{II}(M_s)$
which implies that 
$\rank(M_{s+1}) \le \rank(M_s)$
 independently of the values of
$\rank_{I}(M_s)$ and $\rank_{I}(M_{s+1})$.

Consider now $s$ of type I
so that $\rank_{II}(M_s) = \rank_{II}(M_{s+1})$.
Again, let $j$ be such that $v_{s+1,j} = v_{s,j} + t v_{s,j+1}$, $t > 0$.
For $j' \ne j$, we obtain $v_{s+1,j'} = v_{s,j'}$
and therefore $\tilde v_{s+1,j'} = \tilde v_{s,j'}$.
Thus 
\[ \tilde v_{s+1,j} =
\frac{\omega(v_{s,j})}{\omega(v_{s+1,j})} \tilde v_{s,j} + 
\frac{t \omega(v_{s,j+1})}{\omega(v_{s+1,j})} \tilde v_{s,j+1},
\qquad
\frac{\omega(v_{s,j})}{\omega(v_{s+1,j})} > 0;
\]
implying that $\tilde v_{s+1,j}$ is
an affine combination of $\tilde v_{s,j}$ and $\tilde v_{s,j+1}$.
For $j' \notin \{j-1,j\}$, we get that $w_{s+1,j'} = w_{s,j'}$.
Finally, notice that
$w_{s+1,j}$ is a positive multiple of $w_{s,j}$ and
$w_{s+1,j-1}$ is a positive linear combination of
$w_{s,j-1}$ and $w_{s,j}$. Namely, 
\[ 
w_{s+1,j} = \frac{\omega(v_{s,j})}{\omega(v_{s+1,j})} w_{s,j}, \qquad
w_{s+1,j-1} =
\frac{\omega(v_{s+1,j})}{\omega(v_{s,j})} w_{s,j-1} + t w_{s,j}.  \]
Therefore, up to multiplication by a positive diagonal matrix,
the move from $\tilde M_s$ to $\tilde M_{s+1}$ is 
a positive elementary move.
Thus 
$\rank_{-}(\tilde M_{s+1}) \le \rank_{-}(\tilde M_s)$,
or, equivalently,
$\rank_{I}(M_{s+1}) \le \rank_{I}(M_s)$.
Finally, notice that the inequality $m(M_s) m(M_{s+1}) < 0$
implies that $m(\tilde M_s) m(\tilde M_{s+1}) < 0$
and therefore
$\rank_{I}(M_{s+1}) < \rank_{I}(M_s)$.
\end{proof}

\begin{remark}
\label{remark:loose}
The careful reader might have noticed that, in fact, the proof of
Lemma \ref{lemma:step}  
implies a somewhat sharper claim which is however more difficult to state.
Together with $r(2,n) = 2(n-2)$,
this results in a  stronger statement   
than Theorem \ref{theorem:main}.
But  under such minor improvement the leading term of \eqref{eq:main} will stay the same.   
In particular, these adjustments in our proof will 
 not be sufficient to obtain a linear upper bound for $r(3,n)$ instead of the cubic one given by \eqref{eq:main}.
(The reader will recall that Grassmann convexity conjecture
claims that $r(3,n) = 3(n-3)$, see Equation \eqref{equation:grassmann2}.) 
\end{remark}

\bibliographystyle{crplain}

\nocite{*}

\bibliography{Mybib}

\bigskip

\end{document}